\documentclass[12pt,letterpaper]{article}
\pdfoutput=1

\usepackage{graphics,epsfig,graphicx,color}
\graphicspath{{/EPSF/}{../Figures/}{Figures/}}
\usepackage[caption = false]{subfig}
\usepackage{float}
\usepackage{capt-of}

\usepackage{amssymb,latexsym}

\usepackage{setspace}

\usepackage{amsmath}

\usepackage{mathrsfs,amsfonts}

\usepackage{amsthm,amsxtra}

\usepackage[final]{showkeys}

\usepackage{stmaryrd}

\usepackage{algorithm}
\usepackage{algorithmicx}
\usepackage{algpseudocode}

\usepackage[openbib]{currvita}

\usepackage{enumitem}

\newif\ifPDF
\ifx\pdfoutput\undefined
	\PDFfalse
\else
	\ifnum\pdfoutput > 0
		\PDFtrue
	\else
		\PDFfalse
	\fi
\fi

\ifPDF
	\usepackage{pdftricks}
	\begin{psinputs}
		\usepackage{pstricks}
		\usepackage{pstcol}
		\usepackage{pst-plot}
		\usepackage{pst-tree}
		\usepackage{pst-eps}
		\usepackage{multido}
		\usepackage{pst-node}
		\usepackage{pst-eps}
	\end{psinputs}
\else
	\usepackage{pstricks}
\fi

\ifPDF
	\usepackage[debug,pdftex,colorlinks=true, 
	linkcolor=blue, bookmarksopen=false,
	plainpages=false,pdfpagelabels]{hyperref}
\else
	\usepackage[dvips]{hyperref}
\fi

\usepackage{fancyhdr}

\newtheorem{theorem}{Theorem}[section]

\newcommand{\be}{\begin{equation}}
\newcommand{\ee}{\end{equation}}
\newcommand{\bea}{\begin{eqnarray}}
\newcommand{\eea}{\end{eqnarray}}

\newcommand{\sgn}{\operatorname{sgn}}

\newcommand{\E}{\varepsilon} 
\newcommand{\T}{\tau}

\newcommand{\ado}{\dot a}
\newcommand{\adt}{\ddot a}
\newcommand{\adth}{\dddot a}
\newcommand{\adf}{\ddddot a}

\newcommand{\bdo}{\dot b}
\newcommand{\bdt}{\ddot b}
\newcommand{\bdth}{\dddot b}
\newcommand{\bdf}{\ddddot b}

\newcommand{\ddo}{\dot d}
\newcommand{\ddt}{\ddot d}
\newcommand{\ddth}{\dddot d}
\newcommand{\ddf}{\ddddot d}

\newcommand{\xdo}{\dot x}
\newcommand{\xdt}{\ddot x}
\newcommand{\xdth}{\dddot x}

\newcommand{\Edo}{\dot \E}
\newcommand{\Edt}{\ddot \E}
\newcommand{\Edth}{\dddot \E}
\newcommand{\Edf}{\ddddot \E}

\newcommand{\Tdo}{\dot \T}
\newcommand{\Tdt}{\ddot \T}
\newcommand{\Tdth}{\dddot \T}
\newcommand{\Tdf}{\ddddot \T}

\newcommand{\fdo}{\dot f}
\newcommand{\fdt}{\ddot f}
\newcommand{\fdth}{\dddot f}

\newcommand{\Sdo}{\dot S}
\newcommand{\Sdt}{\ddot S}
\newcommand{\Sdth}{\dddot S}
\newcommand{\Sdf}{\ddddot S}

\newcommand{\Fdo}{\frac{d}{dc}}
\newcommand{\Fdt}{\frac{d^2}{dc^2}}
\newcommand{\Fdth}{\frac{d^3}{dc^3}}
\newcommand{\Fdf}{\frac{d^4}{dc^4}}

\newcommand{\Spo}{S_o^{'}}
\newcommand{\Spt}{S_o^{''}}
\newcommand{\Spth}{S_o^{'''}}
\newcommand{\Spf}{S_o^{''''}}

\newcommand{\Sc}{{\cal S}}

\newcommand{\pushright}[1]{\ifmeasuring@#1\else\omit\hfill$\displaystyle#1$\fi\ignorespaces}
\newcommand{\pushleft}[1]{\ifmeasuring@#1\else\omit$\displaystyle#1$\hfill\fi\ignorespaces}

\newenvironment{keywords}
{\noindent{\bf Key words.}\small}{\par\vspace{1ex}}

\title{{A symmetry breaking transition in the edge/triangle network model}}

\author{
Charles Radin\thanks{Department of Mathematics, University of Texas, Austin, TX 78712; radin@math.utexas.edu}  
\and Kui Ren \thanks{Department of Mathematics and ICES, University of Texas, Austin, TX 78712; ren@math.utexas.edu} 
\and Lorenzo Sadun\thanks{Department of Mathematics, University of Texas, Austin, TX 78712; sadun@math.utexas.edu} 
}

\begin{document}

\maketitle

\begin{abstract}
Our general subject is the emergence of phases, and phase transitions, in large networks subjected to a few variable constraints.  Our main result is the analysis, in the model using edge and triangle subdensities for constraints, of a sharp transition between two phases with different symmetries, analogous to the transition between a fluid and a crystalline solid.
\end{abstract}



\section{Introduction}
\label{SEC:Intro}

Our general subject is the emergence of phases, and phase transitions,
in large networks subjected to a few variable constraints, which we
take as the densities of a few chosen subgraphs. We follow the line of
research begun in~\cite{RS1} in which, for  $k$ given constraints
on a large network one determines a global state on the network, by a
variational principle, from which one can compute a wide range of
global observables of the network, in particular the densities of all
subgraphs. A phase is then a region of constraint values in which 
all these global observables vary smoothly with the
constraint values, while transitions occur when the global state
changes abruptly. (See~\cite{KRRS1,KRRS2,RS1,RS2,RRS,Ko} and the survey~\cite{Rad}.) 
Our main focus is on the transition between two
particular phases, in the system with the two constraints of edge and
triangle density: phases in which the global states have different symmetry.

If one describes networks (graphs) on $n$ nodes by their $n\times n$ adjacency matrices, 
we will be concerned with \emph{asymptotics} as $n\to \infty$, and in particular limits of these
matrices considered as $0$-$1$ valued graphons. (Much of the relevant
asymptotics is a recent development \cite{BCL,BCLSV,LS1,LS2,LS3}; see
\cite{Lov} for an
encyclopedic treatment.)
Given $k$ subgraph densities as constraints, the asymptotic analysis
in~\cite{RS1} leads to the study of one or more $k$-dimensional
manifolds embedded in the infinite dimensional metric space $W$ of
(reduced) graphons, the points in the manifold being the emergent
(global) states of a phase of the network. That is, there are
smooth embeddings, of open connected subsets (called phases) of the
phase space $\Gamma \subset [0,1]^k$ of possible constraint
values, into $W$. The embeddings are obtained from the constrained
entropy density $s(P)$, a real valued function of the constraints
$P\in \Gamma$, through the variational principle~\cite{RS1,RS2}:
$s(P)=\sup\{\Sc(g)\,|\, C(g)=P\}$, the constraints being described by
$C(g)=P$, for instance edge density $\E(g)=P_1$ and triangle density
$\T(g)=P_2$, and $\Sc$ being the negative of the large deviation rate
function of Chatterjee-Varadhan~\cite{CV}.

The above framework is modeled on that of statistical physics in which
the system is the simultaneous states of many interacting particles,
the constraints are the invariants of motion (mass and energy
densities for simple materials) and the global states, `Gibbs states',
can be understood in terms of conditional probabilities~\cite{Ru2}. The relevant variational principle was proven in~\cite{Ru1}.
In contrast, for large graphs the constrained entropy optima turn out
to be much easier to analyze than Gibbs states. First, it has been
found that in all known cases entropy optima are `multipodal',
i.e.\ for each phase in any graph model they lie in some $M$
dimensional manifold in $W$ corresponding to a decomposition of all
the nodes into $M$ equivalence classes~\cite{KRRS1,KRRS2,Ko}. This brings
the embedding down into a fixed finite dimension, at least within each
phase, in place of the infinite dimensions of $W$. In practice, for
$k=2$ this often allows one to find an embedding into 4 dimensions, so
it only remains to understand how our 2 dimensional surface sits in
those 4 dimensions. The main goal of this paper is to study this near
a particular transition for edge/triangle constraints.
\begin{figure}[htbp]
\centering
\includegraphics[width=0.4\textwidth,height=0.35\textwidth]{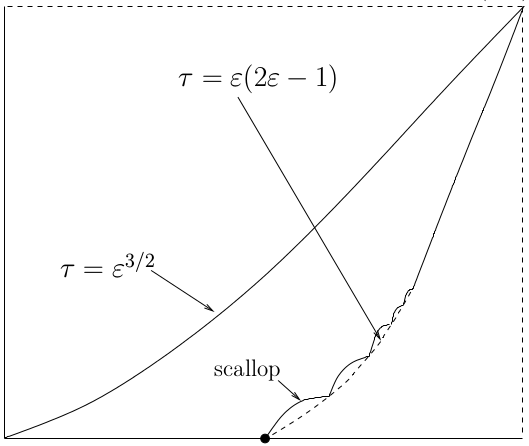}\hskip 1cm
\includegraphics[width=0.4\textwidth,height=0.35\textwidth]{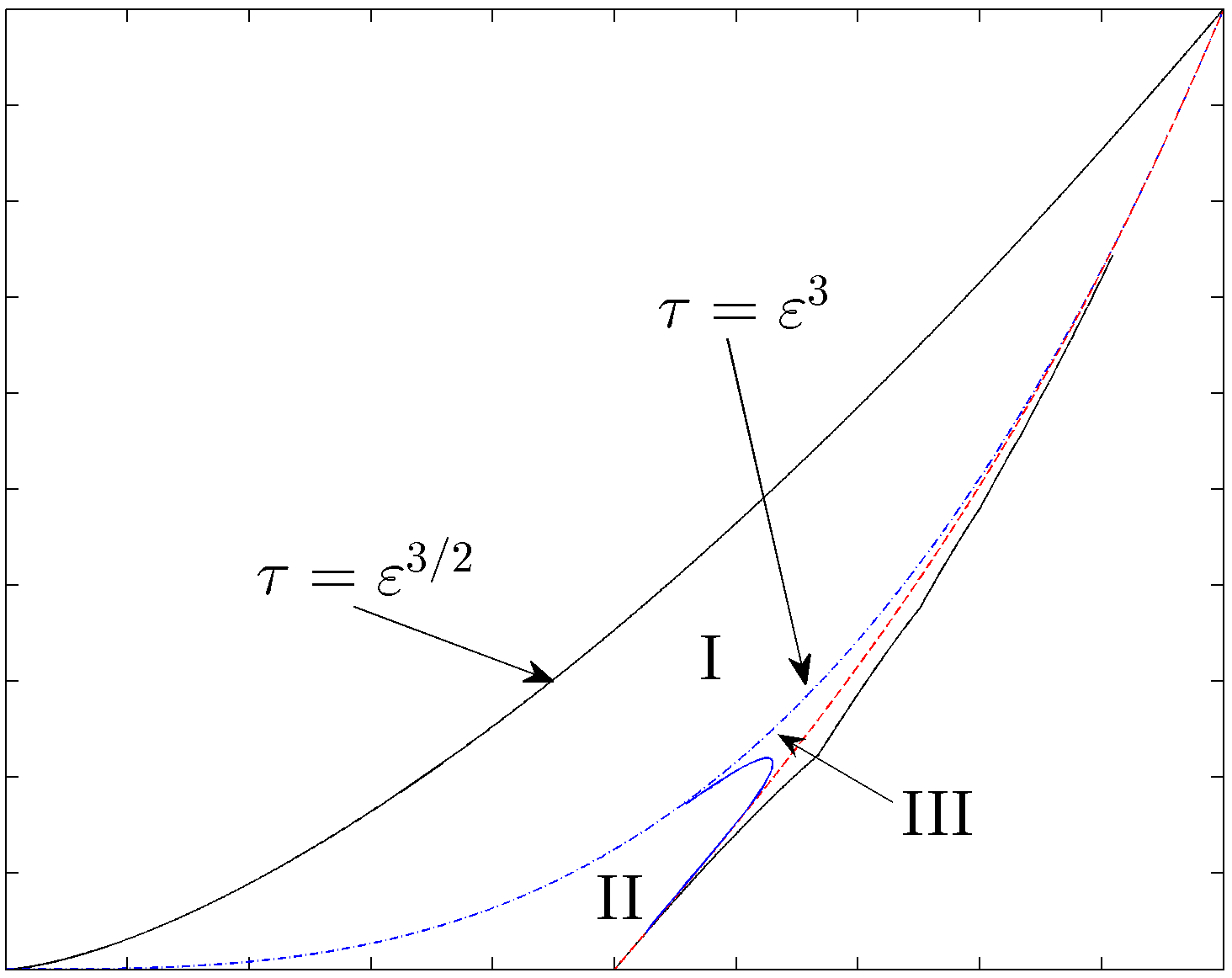}
\put(-105,-8){{\scriptsize $(1/2, 0)$}}
\put(-110,-17){{\scriptsize \rm edge density $\varepsilon$}}
\put(-325,-8){{\scriptsize $(1/2, 0)$}}
\put(-330,-17){{\scriptsize \rm edge density $\varepsilon$}}
\put(-200,105){\rotatebox{270}{\scriptsize\bf \rm triangle density $\tau$}}
\put(-415,105){\rotatebox{270}{\scriptsize\bf \rm triangle density $\tau$}}
\caption{Left: Boundary of the phase space for subgraph constraints edges and triangles. The scalloped region is exaggerated for better visualization; Right: Boundaries of phases I, II and III.}
\label{scallop}
\end{figure}
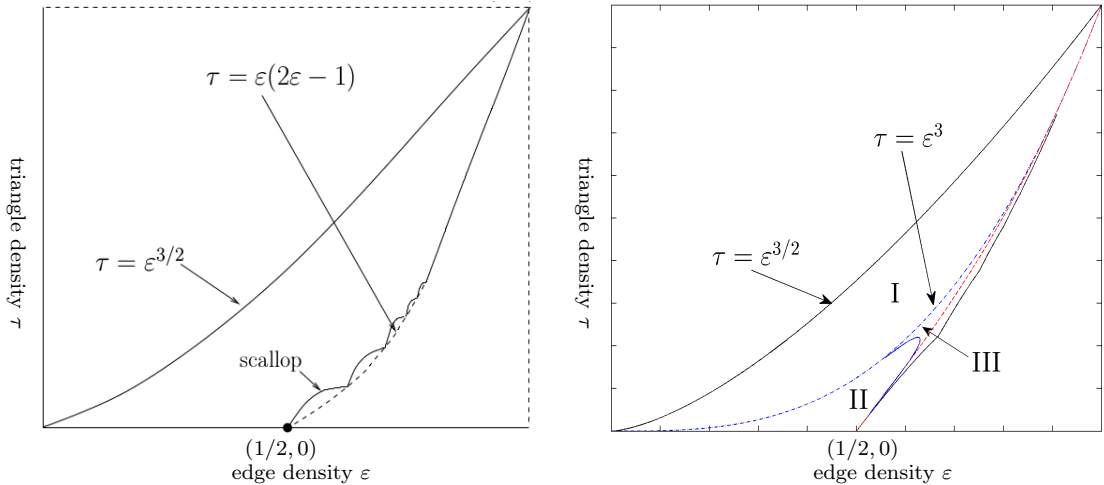 
The phase space $\Gamma$ of possible values of edge/triangle
constraints is the interior of the scalloped triangle of Razborov~\cite{Ra}, sketched in
Figure~\ref{scallop}. Extensive simulation in this model~\cite{RRS} shows the
existence of several phases, labeled $I$, $II$ and $III$ in Figure~
\ref{scallop}. 

The main focus of this paper is the transition between phases $II$ and
$III$. There is good simulation evidence, but there is still no proof
of this transition. What we will do is \emph{assume}, as seen in simulation,
that all entropy optimal graphons for edge/triangle densities near the
transition are bipodal, that is given by a piecewise constant function
of the form:
\begin{equation}\label{general-bipodal}
 g(x,y) = \begin{cases} a & x,y < c, \cr 
  d & x<c<y, \cr d & y<c<x, \cr b & x,y >
  c. \end{cases} 
\end{equation}
{{Bipodal graphons are generalizations of {bipartite graphons}, in which $a=b=0$.}}
 Here $c,a,d$ and $b$ are
constants taking values between 0 and 1. We do not assume these entropy
optimizers are unique. But from the bipodal assumption we can prove
uniqueness, and also derive, as shown
already in~\cite{RRS}, the equations determining the transition curve, and
prove that the optimal graphons for edge/triangle constraints  $(\E,\T)$ to the left of the
transition, have the form:
\begin{equation}\label{EQ:Graphon2}
g(x,y) = \begin{cases} \E - (\E^3-\T)^{1/3} & x,y < 1/2 \hbox{ or } x,y > 1/2 \cr
\E + (\E^3-\T)^{1/3} & x< \frac{1}{2} < y \hbox{ or } y < \frac{1}{2} < x,
\end{cases}
\end{equation}
which is highly symmetric in the sense that $c=1/2$ and
$a=b$. The main result of this paper is the derivation of
the lower order terms in $\E$ and $\T$ of the bipodal parameters
$a,b,c,d$ as
the constraints move to the right of the transition, i.e.\ we
determine how the symmetry is broken at the transition. Before we get
into details we should explain the connection between
this study of emergent phases and their transitions, and
other work on random graphs using similar terminology.

The word phase is used in many ways in the literature, but `emergent phase' is more
specific and refers to a coherent, large scale description of a system
of many similar components, in the following sense.
Consider large graphs, thought of as systems of many edges on a fixed
number of labeled nodes. To work at a large scale means to be
primarily concerned with global features or observables, for instance densities of
subgraphs. 

Finite graphs are fully
described (i.e.\ on a small scale) by their adjacency matrices. In the
graphon formalism these have a scale-free description in which each
original node is replaceable by a cluster of $m$ nodes, which makes for
a convenient analysis connecting large and small scale.

Phases are concerned with large graphs under a small number $k$ of
variable global constraints, for instance constrained by the 2
subdensities of edges and triangles. We say one or more phases \emph{emerge}
for such systems, corresponding to one or more open connected subsets
of parameter values, if: 1) there are \emph{unique} global
states (graphons) associated with the sets of constraint
values; 2) the correspondence defines a smooth $k$-dimensional surface
in the infinite dimensional space of states.

Note that not all achievable parameter values (the phase space) belong
to phases, in particular the boundary of the phase space does not. In
fact emergent phases are interesting in large part because they
exhibit interesting (singular) boundary behavior in the interior of
the phase space. For graphs we see this in at least two ways familiar
from statistical mechanics. In the edge/2-star model there is only one
phase but there are achievable parameter values with multiple states
associated (as in the liquid/gas transition); see Figure~\ref{edge-2star}.
\vskip.2truein
\begin{figure}[htbp]
\center{\includegraphics[width=2.5in]{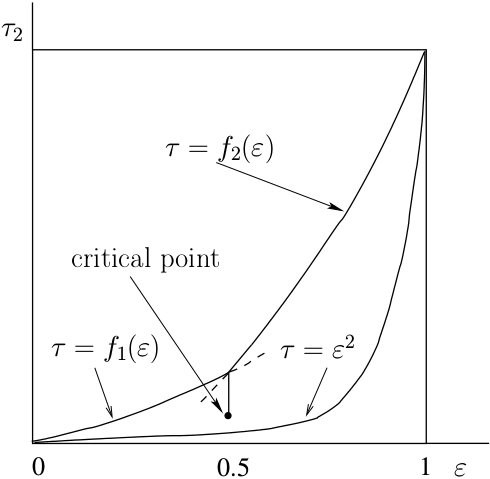}}
\caption{Transition ending in a critical point for edge/2-star model}
\label{edge-2star}
\end{figure}
In most models
constraint values on the Erd\"os-R\'enyi curve -- parameter values
corresponding to iid edges -- provide unique states
which however do not vary smoothly across the curve, and this provides 
boundaries of phases~\cite{RS1}. For edge/triangle constraints there is
another phase boundary associated with a loss of symmetry (as in the
transition between fluid and crystalline solid). In all models studied
so far phases make up all but a lower dimensional subset of the phase
space but this may not be true in general.

From an alternative vantage random graphs are commonly used to model
some particular network, the idea being to fit parameter values in a
parametric family of probability distributions on some space of
possible networks, so as to get a manageable probabilistic picture of
the one target network of interest; see~\cite{CD,N} and references
therein. The one parameter Erd\"os-R\'enyi (ER) family is often used
this way for percolation, as are multiparameter generalizations such
as stochastic block models and exponential random graph models
(ERGMs). Asymptotics is only a small part of such research, and for
ERGMs it poses some difficulties as is well described in~\cite{CD}

Newman et al considered phase transitions for ERGMs; see~\cite{PN} for
a particularly relevant example. Chatterjee-Varadhan
introduced in~\cite{CV} a powerful asymptotic formalism for random graphs
with their large deviations theorem, within the previously developed graphon
formalism. This was applied to ERMGs in~\cite{CD}, which
was then married with Newman's phase analysis in~\cite{RY} and numerous
following works.

We now contrast this use of the term phase with the emergent phase analysis
discussed above. In the latter a phase corresponds to an embedding of
part of the parameter space into the emergent (global) states of
the network, so transitions could be interpreted as singular behavior
\emph{as those global network states varied}. For ERGMs it was shown 
in~\cite{CD} that the analogous map from the parameter space $\Gamma$ into $W$, determined by optimization of a free energy rather
than entropy, is often many-to-one; for instance \emph{all} states of
the 2-parameter edge/2-star ERGM model actually belong to the
1-parameter ER family. Therefore transitions in
an ERGM are naturally understood as asymptotic behavior of the \emph{model}
{under variation of model parameters, rather than singular
  behavior among naturally varying global states of a network};
in keeping with its natural use, a transition in an ERGM says more
about the model than about constrained networks.

\section{Outline of the calculation}
\label{SEC:Outline}

The purpose of this calculation is to understand the transition, in the 
edge/triangle model, between phases $II$ and $III$ in Figure~\ref{scallop}. 
We take
as an assumption that our optimizing graphon at fixed constraint
$(\E,\T)$ is bipodal as in (\ref{general-bipodal}), and with $c$ being 
the size of the first cluster of nodes. We denote this graphon by $g_{abcd}$; see Figure~\ref{FIG:Setup}.
With the notation
\be
S_o(z)=-\frac{1}{2}[z\ln(z)+(1-z)\ln(1-z)],\ 0\le z\le 1
\ee
and
\be
\Sc(g)=\int_{[0,1]^2}S_o[g(x,y)]\,dxdy,
\ee
to obtain $s(P)$ we maximize $\Sc(g_{abcd})$ by varying the four parameters $(a,b,c,d)$
while holding $P=(\E,\T)$ fixed. 

It is not hard to check that the symmetric graphon with $a=b=\E+(\T-\E^3)^{1/3}$,
$c=\frac12$ and $d=\E-(\T-\E^3)^{1/3}$ is always a stationary point of the
functional $\Sc$. But is it a maximum? Near the choices of 
$(\E,\T)$ where it ceases to be a maximum, what does the actual maximizing
graphon look like? 

We answer this by doing our constrained optimization in stages. First we
fix $c$ and vary $(a,b,d)$ to maximize $\Sc(g)$ subject to the constraints
on $\E$ and $\T$. Call this maximum value $S(c)$. 
Since the graphon with parameters 
$(b,a,1-c,d)$ is equivalent to $(a,b,c,d)$, $S(c)$ is an even function of 
$c-\frac12$. We expand this function in a Taylor series:
\be \label{Taylor-S}  S(c) = S(\frac12) + \frac12 \ddot S(\frac12)(c-\frac12)^2 +
\frac{1}{24} \ddddot S(\frac12)(c-\frac12)^4 + \cdots \ee
where dots denote derivatives with respect to $c$. 
  
If $\Sdt(\frac12)$ is negative, then the symmetric graphon is stable 
against small changes in $c$. If $\Sdt(\frac12)$ becomes positive
(as we vary $\E$ and $\T$), then the local maximum of $S(c)$ at $c=\frac12$
becomes a local minimum. As long as $\Sdf(\frac12)<0$, 
new local maxima will appear at $c \approx \frac12
\pm \sqrt{-6 \Sdt(\frac12)/\Sdf(\frac12)}$. In this case, 
as we pass through the phase transition, we should expect the optimal 
$|c-\frac12|$ to be exactly zero on one side of the transition line
(i.e., in the symmetric phase), and to vary as the square root of the distance to the transition line on
the other (asymmetric) side. 

If $\Sdf(\frac12)$ is positive when $\Sdt(\frac12)$ passes through
zero, something very different happens. Although $c=\frac12$ is a 
{\em local} maximum whenever $\Sdt(\frac12)<0$, there are local maxima
elsewhere, at locations determined largely by the higher order terms in
the Taylor expansion~\eqref{Taylor-S}. When $\Sdt(\frac12)$ passes
below a certain threshold, one of these local maxima will have a higher
entropy than the local maximum at $c=\frac12$, and the optimal value of 
$c$ will change discontinuously. 

The calculations below give analytic formulas for $\Sdt(\frac12)$ and 
$\Sdf(\frac12)$ as functions of $\E$ and $\T$. We can then evaluate
these formulas numerically to fix the location of the phase transition curve
and determine which of the previous two paragraphs
more accurately describes the phase transition near a given point on 
that curve.  We then compare these results to numerical sampling 
that is done {\em without} the simplifying assumption that optimizing
graphons are bipodal. 

Our results  indicate that
\begin{itemize}
\item The assumption of bipodality is justified, and  
\item $\ddddot S(\frac12)$ remains negative on the phase transition curve,
implying that $c$ does not change discontinuously. Rather, $|c-\frac12|$
goes as the square root of the 
distance to phase transition curve, in the 
asymmetric bipodal phase.  
\end{itemize}

\section{Exact formulas for $\ddot S(\frac12)$ 
and $\ddddot S(\frac12)$}
\label{SEC:Derivation}

We now present the details of our perturbation calculations in detail. 

\subsection{Varying $(a,b,d)$ for fixed $c$.}

The first step in the calculation is to derive the variational equations
for maximizing $\Sc(g_{abcd})$ for fixed $(\E, \T, c)$. We first express the 
edge and triangle densities and $S=\Sc(g_{abcd})$ as functions of $(a,b,c,d)$:
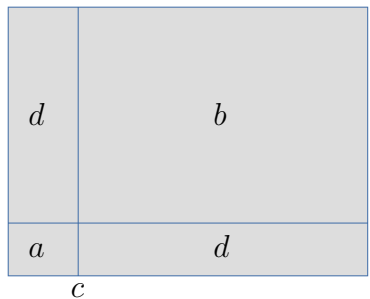
\begin{figure}[ht]
\centering
\includegraphics[angle=0,width=0.3\textwidth]{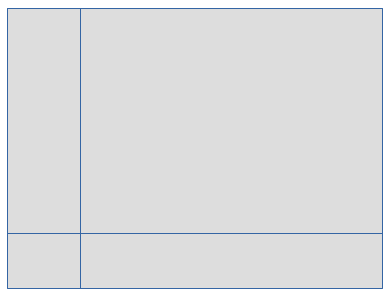} 
\put(-130,10){$a$}
\put(-60,10){$d$}
\put(-130,60){$d$}
\put(-60,60){$b$}
\put(-114,-5){$c$}
\caption{The parameter set $(a, b, c, d)$ of a bipodal graphon.}
\label{FIG:Setup}
\end{figure}

\begin{eqnarray}
\E&=&c^2a+2c(1-c)d+(1-c)^2b \cr 
\T &=& c^3a^3+3c^2(1-c)ad^2+3c(1-c)^2bd^2+(1-c)^3b^3
\cr
S&=&c^2S_o(a)+(1-c)^2S_o(b)+2c(1-c)S_o(d)
\end{eqnarray}
Next we compute the gradient of these quantities with respect to $(a,b,d)$, where $\nabla_3$ is an abbreviation
for $(\partial_a, \partial_b, \partial_d)$:

\begin{eqnarray}
\nabla_3 \, \E = &&(c^2,(1-c)^2,2c(1-c))
\cr
\nabla_3\, \T = &&(3c^3a^2+3c^2(1-c)d^2,3(1-c)^3b^2+3c(1-c)^2d^2, \cr 
&& \qquad \qquad \qquad 6c^2(1-c)ad +6c(1-c)^2bd)
\cr
\nabla_3\,
S = &&(c^2\Spo(a),(1-c)^2\Spo(b),2c(1-c)\Spo(d))
\end{eqnarray}

At a maximum of $S$ (for fixed $c$, $\E$ and $\T$), $\nabla_3\,S$ must be a linear combination of 
$\nabla_3\,\E$ and $\nabla_3\,\T$ so
\begin{eqnarray}
0 = &&\det \begin{pmatrix}
\nabla_3\,\E \cr
 \nabla_3\,\T \cr
\nabla_3\,S \cr
\end{pmatrix}
\cr
= &&c^2(1-c)^26c(1-c)\cr
&&\times \det\begin{pmatrix}
1&1&1\cr
a^2c+d^2(1-c)&d^2c+b^2(1-c)&cad+(1-c)bd\cr
\Spo(a)&\Spo(b)&\Spo(d)\cr
\end{pmatrix} \eea
Expanding the determinant and dividing by $6c^3(1-c)^3$, we obtain 
\bea 0 = f(a,b,c,d) & :=&\phantom{+} \Spo(a)[cd(a-d)-(1-c)b(b-d)] \cr 
&&+\Spo(b)[ca(a-d)-(1-c)d(b-d)]\cr
&&+\Spo(d)[c(d^2-a^2)+(1-c)(b^2-d^2)]. \eea

\subsection{Strategy for varying $c$}

We just showed how optimizing the entropy for fixed $\E$, $\T$, and $c$ 
is equivalent to setting $f=0$. From now on we treat $f=0$ as
an additional constraint. The three constraint 
equations $\E=\E_o$, $\T=\T_o$, $f=0$ 
define a curve in $(a,b,d,c)$
space, which we can parametrize by $c$. Since $\E$, $\T$ and $f$ are constant, derivatives 
of these quantities along the curve must be zero. By evaluating these derivatives at $c=1/2$,  
we will derive
formulas for $\ado(1/2)$, $\bdo(1/2)$, etc., where a dot denotes a derivative with respect 
to $c$ along this curve. These values then determine  
$\ddot S(\frac12)$ and $\ddddot S(\frac12)$. 

We also make use of symmetry. The parameters $(a,b,c,d)$ and $(b,a,1-c,d)$ describe the same reduced graphon. Thus, 
if the conditions $\E=\E_0$, $\T = \T_0$, $f=0$ trace out a unique curve in $(a,b,c,d)$ space, then we must have
\bea \label{sym-eq}
\ado(1/2)+ \bdo(1/2) &=&  \adt(1/2)-\bdt(1/2) = \adth(1/2)+\bdth(1/2)
= \adf(1/2)-\bdf(1/2) \cr &=&\ddo(1/2)=\ddth(1/2)=0,
\eea 
since all of these quantities are odd under the interchange $(a,b,c,d) \leftrightarrow (b,a,1-c,d)$. 

In fact, it is possible to derive the relations (\ref{sym-eq}), and similar relations for higher derivatives, 
{\em without} assuming uniqueness of the curve. We proceed by induction on
the degree $k$ of the derivatives involved.
At 0-th order we already have that $a(1/2)=b(1/2)$. 

The $k$th derivative of $\E$ and $\T$, evaluated at $c=1/2$, are
\bea \E^{(k)}(1/2) = &&\frac{1}{4} \left ( a^{(k)}(1/2) + b^{(k)}(1/2) + 2 d^{(k)}(1/2) \right )\cr
&&+ \hbox{(lower order derivatives of $a,b,d$)} \cr 
\T^{(k)}(1/2) &=& \frac38 \left ( (a(1/2)^2+d(1/2)^2) a^{(k)}(1/2) + (b(1/2)^2+d(1/2)^2) b^{(k)}(1/2)\right) \cr && + \frac34 \left (a(1/2)+b(1/2))d(1/2) d^{(k)}(1/2 \right ) \cr && \qquad + \hbox{(lower order derivatives and their products)}
\eea

When $k$ is odd, the lower-order terms vanish by induction, 
and we are left with 
\bea \label{vanish-1} 0 &=& 
\phantom{( (a(1/2)^2 + d(1/2)^2)}
a^{(k)}(1/2) + b^{(k)}(1/2) + \phantom{4a(1/2)d(1/2)}
2 d^{(k)}(1/2) \cr
0 &= & (a(1/2)^2 + d(1/2)^2) (a^{(k)}(1/2) +b^{(k)}(1/2) ) + 4a(1/2)d(1/2) d^{(k)}(1/2), \eea
since $b(1/2)=a(1/2)$. The matrix 
$\begin{pmatrix} 1 & 2 \cr a(1/2)^2 + d(1/2)^2 & 4a(1/2)d(1/2) \end{pmatrix}$ is non-singular, having
determinant $-2(d(1/2)-a(1/2))^2$, so
$a^{(k)}(1/2)+ b^{(k)}(1/2)= d^{(k)}(1/2) = 0$. 

When $k$ is even, the $k$th derivative of $f$, evaluated at $c=1/2$,  is of the form
\bea \hbox{(Even function w.r.t the interchange)}
&\times&(a^{(k)}(1/2)-b^{(k)}(1/2))
\cr 
&+& \hbox{lower order derivatives}\eea
As before, the lower order terms vanish by induction, and the even function 
is nonzero, so we get $a^{(k)}(1/2)=b^{(k)}(1/2)$. 

Henceforth we will freely use the relations in~\eqref{sym-eq} to 
simplify our expressions. 
  
Here are the steps of the calculation:
\begin{itemize}
\item Setting $\fdo(1/2)=0$ determines $\ado(1/2)$. 
\item Setting $\Edt(1/2)=\Tdt(1/2)=0$ determines $\adt(1/2)$ 
and $\ddt(1/2)$ (in terms of $\ado(1/2)$). 
\item Setting $\fdth(1/2)=0$ determines $\adth(1/2)$.
\item Setting $\Edf(1/2)=\Tdf(1/2)=0$ determines $\adf(1/2)$ and $\ddf(1/2)$. 
\item Once all derivatives of $(a,b,d)$ up to 4th order are 
evaluated at $c=1/2$, we explicitly compute $\ddot S(1/2)$ and 
$\ddddot S(1/2)$. 
\end{itemize}

\subsection{Derivatives of $\E$ and $\T$}

The edge density and its first four derivatives are:
\bea
\E&=&c^2a+2(c-c^2)d+(1-c)^2b
\cr
\Edo &=& c^2\ado +2(c-c^2)\ddo + (1-c)^2\bdo +2ca-2(1-c)b+2(1-2c)d
\cr
\Edt &=& c^2\adt +(1-c)^2\bdt +2(c-c^2)\ddt +4c\ado -4(1-c)\bdo
+4(1-2c)\ddo +2a+2b-4d
\cr
\Edth &=&  c^2\adth +(1-c)^2\bdth +2c(1-c)\ddth +6c\adt -6(1-c)\bdt
+6(1-2c)\ddt +6\ado +6\bdo -12\ddo
\cr
\Edf &=& c^2 \adf +(1-c)^2 \bdf +2c(1-c)\ddf +8c\adth -8(1-c)\bdth
+8(1-2c)\ddth \cr && + 12\adt +12\bdt -24\ddt  
\eea

Evaluating at $c=1/2$ and applying the symmetries (\ref{sym-eq}) gives:
\bea \label{Edots} \Edo(1/2) & = & \frac{1}{4} \left(\ado(1/2)+\bdo(1/2)+2\ddo(1/2) \right )=0 \cr 
\Edt(1/2) &=& \frac14 \left (\adt(1/2)+\bdt(1/2) +2\ddt(1/2) \right)+4\ado(1/2)+4(a(1/2)-d(1/2)) \cr 
& =& \frac12 \left ( \adt(1/2) +\ddt(1/2) \right )+4\ado(1/2)+4(a(1/2)-d(1/2)) \cr 
\Edth(1/2) & =& \frac14 \Big (\adth(1/2)+\bdth(1/2)+2\ddth(1/2) \Big ) =0 \cr 
\Edf(1/2) &=& \frac14 \Big (\adf(1/2) +\bdf(1/2)+2\ddf(1/2) \Big )+8\adth(1/2)+24\adt(1/2) -24\ddt(1/2) \cr 
&=& \frac12 \Big (\adf(1/2)+\ddf(1/2) \Big )+8\adth(1/2) +24\left (\adt(1/2) -\ddt(1/2)\right )
\eea

Next we compute the triangle density and its derivatives:
\bea
\T&=& c^3a^3+3c^2(1-c)ad^2+3c(1-c)^2 bd^2+(1-c)^3b^3
\cr
\Tdo&=& c^3\frac{d}{dc}(a^3)+3c^2(1-c)\frac{d}{dc}(ad^2)
+3c(1-c)^2\frac{d}{dc}(bd^2)+(1-c)^3\frac{d}{dc}(b^3)\cr
&& +3c^2a^3+3(2c-3c^2)ad^2+3(1-4c+3c^2)bd^2-3(1-c)^2b^3\cr 
\Tdt &=& c^3\frac{d^2}{dc^2}(a^3)+3c^2(1-c)\frac{d^2}{dc^2}(ad^2)+3c(1-c)^2\frac{d^2}{dc^2}(bd^2)\cr
&&+(1-c)^3\frac{d^2}{dc^2}(b^3)+6c^2\frac{d}{dc}(a^3)+6(2c-3c^2)\frac{d}{dc}(ad^2)\cr
&&+6(1-4c+3c^2)\Fdo(bd^2)-6(1-c)^2\Fdo(b^3)\cr
&&+6ca^3+6(1-3c)ad^2+6(3c-2)bd^2+6(1-c)b^3 \cr 
\Tdth &=& c^3\frac{d^3}{dc^3}(a^3)+3c^2(1-c)\frac{d^3}{dc^3}(ad^2)+3c(1-c)^2\frac{d^3}{dc^3}(bd^2)\cr
&&+(1-c)^3\frac{d^3}{dc^3}(b^3)+9c^2\frac{d^2}{dc^2}(a^3)+9(2c-3c^2)\frac{d^2}{dc^2}(ad^2)\cr
&&+9(1-4c+3c^2)\frac{d^2}{dc^2}b(d^2)-9(1-c)^2\frac{d^2}{dc^2}(b^3)+18c\frac{d}{dc}(a^3)\cr
&&+18(1-3c)\frac{d}{dc}(ad^2)+18(3c-2)\frac{d}{dc}(bd^2)+18(1-c)\frac{d}{dc}(b^3)\cr
&&+6a^3-18ad^2+18bd^2-6b^3 \cr
\Tdf&=& c^3\frac{d^4}{dc^4}(a^3)+3c^2(1-c)\Fdf(ad^2)+3c(1-c)^2\Fdf(bd^2)+(1-c)^3\Fdf(b^3)\cr
&&+12c^2\Fdth(a^3)+12(2c-3c^2)\Fdth(ad^2)+12(1-4c+3c^2)\Fdth(bd^2)\cr
&&-12(1-c)^2\Fdth(b^3)+36c\Fdt(a^3)+36(1-3c)\Fdt(ad^2)\cr
&&+36(3c-2)\Fdt(bd^2)+36(1-c)\Fdt(b^3)+24\Fdo(a^3)\cr
&&-72\Fdo(ad^2)+72\Fdo(bd^2)-24\Fdo(b^3)
\eea

Once again we evaluate at $c=1/2$, making use of symmetry to simplify terms:
\bea
\Tdo(1/2)& =& (1/8)[\Fdo(a^3)+3\Fdo(ad^2)+3\Fdo(bd^2)+\Fdo(b^3)]
+(3/4)(a^3+ad^2-bd^2-b^3) \cr 
&=& (1/8)[\Fdo(a^3)+3\Fdo(ad^2)+3\Fdo(bd^2)+\Fdo(b^3)] \cr 
& = & \frac38 [(a^2+d^2)(\ado +\bdo) + 4ad \ddo] = 0 
\eea
where all quantities on the right hand side are evaluated at $c=1/2$.
Continuing to higher derivatives, 
\bea
\Tdt(1/2)&=&(1/8)[\Fdt(a^3)+3\Fdt(ad^2)+3\Fdt(bd^2)+\Fdt(b^3)]\cr
&&+(3/2)[\Fdo(a^3)+\Fdo(ad^2)-\Fdo(bd^2)-\Fdo(b^3)]\cr
&&+ 3[a^3 -ad^2-bd^2+b^3] \cr 
&=& (1/4)[\Fdt(a^3)+3\Fdt(ad^2)]
+3[\Fdo(a^3)+\Fdo(ad^2)]\cr
&+&6(a^3-ad^2).
\eea

\be \Tdth=0
\ee

\bea
\Tdf(1/2)&=&(1/8)[\Fdf(a^3)+3\Fdf(ad^2)+3\Fdf(bd^2)+\Fdf(b^3)]\cr
&&+3[\Fdth(a^3)+\Fdth(ad^2)-\Fdth(bd^2)-\Fdth(b^3)]\cr
&&+18[\Fdt(a^3)-\Fdt(ad^2)-\Fdt(bd^2)+\Fdt(b^3)]\cr
&&+24[\Fdo(a^3)-3\Fdo(ad^2)+3\Fdo(bd^2)-\Fdo(b^3)]\cr
&=&(1/4)[\Fdf(a^3)+3\Fdf(ad^2)]+6[\Fdth(a^3)+\Fdth(ad^2)]\cr
&&+36[\Fdt(a^3)-\Fdt(ad^2)]+48[\Fdo(a^3)-3\Fdo(ad^2)]
\eea

To continue further we must expand the derivatives of $a^3$ and $ad^2$:

\bea
\Fdo(a^3)&=&3a^2\ado
\cr
\Fdt(a^3) &=& 3a^2\adt+6a\ado^2
\cr
\Fdth(a^3)&=& 3a^2\adth+18a\ado\adt +6\ado^3
\cr
\Fdf(a^3)&=& 3a^2\adf+24a\ado\adth+18a\adt^2+36\ado^2 \adt
\eea

\bea
\Fdo(ad^2)&=& \ado d^2+2ad\ddo
\cr
\Fdt(ad^2)&=&
\adt d^2+4\ado d\ddo+2a\ddo^2+2ad\ddt
\cr
\Fdth(ad^2)&=& \adth d^2 +6\adt d\ddo +6\ado d\ddt +6\ado \ddo^2
+2ad \ddth + 6a\ddo \ddt
\cr
\Fdf(ad^2)&=&\adf d^2+8\adth d\ddo+12\adt d\ddt+12\adt \ddo^2\cr
&&+ 8\ado d\ddth+24\ado \ddo \ddt+2ad\ddf+8a\ddo \ddth+6a\ddt^2
\eea

At $c=1/2$, $\ddo=\ddth=0$, so this simplifies to:
\bea
\Fdo(ad^2)(1/2)&=&\ado d^2
\cr
\Fdt(ad^2)(1/2)&=&\adt d^2+2ad\ddt
\cr
\Fdth(ad^2)(1/2)&=&\adth d^2+6\ado d \ddt
\cr
\Fdf(ad^2)(1/2) &=&\adf d^2+12\adt d \ddt \cr && +6a(1/2) \ddt^2+2ad \ddf,
\eea
all evaluated at $c=1/2$.

Plugging back in, this yields
\bea
\Tdt(1/2) &=&(1/4)[\Fdt(a^3)+3\Fdt(ad^2)]+3[\Fdo(a^3)+\Fdo(ad^2)]+6(a^3-ad^2)\cr
&=&(1/4)[3a^2 \adt+6a\ado^2+3\adt d^2 +6ad\ddt]
+3[3a^2\ado+\ado d^2]+6a(a^2-d^2)\cr
&=&(3/4)[(a^2+d^2)\adt+2ad\ddt]+(3/2)a\ado^2
+3(3a^2+d^2)\ado
+6a(a^2-d^2)
\eea
and \newpage
\bea \label{Tdf-eq}
\Tdf(1/2) &=&(1/4)[\Fdf(a^3)+3\Fdf(ad^2)]+6[\Fdth(a^3)+\Fdth(ad^2)]\cr
&&+36[\Fdt(a^3)-\Fdt(ad^2)]+48[\Fdo(a^3)-3\Fdo(ad^2)]\cr
&=&(1/4)[3a^2\adf+24a\ado \adth+18a\adt^2 +36\ado^2 \adt\cr
&&+3d^2\adf+36d\ddt \adt+18a\ddt^2+6ad\ddf]\cr
&&+6[3a^2\adth+18a\ado\adt+6\ado^3+\adth d^2+6\ado d\ddt]\cr
&&+36[3a^2\adt+6a\ado^2-d^2\adt -2ad\ddt]\cr
&&+48[3a^2\ado-3d^2\ado],
\eea
where all quantities on the right hand side of these equations are evaluated at $c=1/2$.

\subsection{Derivatives of $f(a,b,c,d)$}

Taking the derivative of 
\bea
f(a,b,c,d) &=&\phantom{+}\Spo(a)[cd(a-d)-(1-c)b(b-d)]\cr
&&+\Spo(b)[ca(a-d)-(1-c)d(b-d)]\cr
&&+\Spo(d)[c(d^2-a^2)+(1-c)(b^2-d^2)]
\eea
with respect to $c$ (and applying the chain rule) then yields 
\bea
\fdo&=&\Spo(a)[d(a-d)+b(b-d)+c(-2d\ddo+a\ddo+d\ado)
-(1-c)(2b\bdo-b\ddo-d\bdo)]\cr
&&+\Spo(b)[a(a-d)+d(b-d)+c(2a\ado-\ado d-d\ado)
-(1-c)(b\ddo+d\bdo-2d\ddo)]\cr
&&+\Spo(d)[(d^2-a^2)-(b^2-d^2)+2c(d\ddo-a\ado)+2(1-c)(b\bdo-d\ddo)]\cr
&&+\ado\Spt(a)[cd(a-d)-(1-c)b(b-d)]\cr
&&+\bdo\Spt(b)[ca(a-d)-(1-c)d(b-d)]\cr
&&+\ddo\Spt(d)[c(d^2-a^2)+(1-c)(b^2-d^2)]
\eea

At $c=1/2$ this simplifies to 
\bea
\fdo(1/2)&=&\Spo(a)[a^2-d^2+(1/2)d\ado+(1/2)d\bdo-b\bdo]\cr
&&+\Spo(b)[a^2-d^2+a\ado-(1/2)d\ado -(1/2)d\bdo]\cr
&&+\Spo(d)[2d^2-2a^2-a\ado+b\bdo]\cr
&&+\ado\Spt(a)[-(1/2)(a-d)^2]+\bdo\Spt(b)[(1/2)(a-d)^2]\cr
&=&\Spo(a)[2a^2-2d^2+2a\ado]+\Spo(d)[2d^2-2a^2-2a\ado]\cr
&&-\ado \Spt(a)(a-d)^2,
\eea
where all terms on the right are evaluated at $c=1/2$. Setting this equal to zero then yields
\be
0= 2(a^2-d^2)[\Spo(a)-\Spo(d)]
+\ado [2a(\Spo(a)-\Spo(d))-\Spt(a)(a-d)^2],
\ee

\be \label{ado-eq}
\ado(1/2)=\frac{2(d^2-a^2)[\Spo(a)-\Spo(d)]}{2a[\Spo(a)-\Spo(d)]-\Spt(a)(a-d)^2},
\ee
evaluated at $c=1/2$.

Before computing higher derivatives of $f$, we go back and use this value of $\ado(1/2)$ to 
compute $\adt(1/2)$ and $\ddt(1/2)$:

\be \label{Edt-eq}
0=2 \Edt(1/2)=\adt+\ddt+8\ado+8(a-d),
\ee
where all quantities are computed at $c=1/2$. Likewise, 
\be \label{Tdt-eq}
0=\frac43 \Tdt=(a^2+d^2)\adt +2ad\ddt+2a\ado^2
+4(3a^2+d^2)\ado +8a(a^2-d^2).
\ee
However, $2ad$ times equation (\ref{Edt-eq}) is 
\be
0=2ad\adt +2ad\ddt+16ad\ado+16ad(a-d)
\ee
Subtracting these equations gives 
\bea
0&=&(a-d)^2\adt+2a\ado^2+(12a^2+4d^2 -16ad)\ado
+8a^3 -8ad^2 -16a^2d+16ad^2
\cr
&=&(a-d)^2 \adt +2a\ado^2 +4(a-d)(3a-d)\ado +8a(a-d)^2
\eea

Solving for $\adt$ then gives:
\be \label{adt-eq1}
\adt(1/2)=\frac{-2a(\ado^2 -4(a-d)(3a-d)\ado -8a(a-d)^2}{(a-d)^2},
\ee
where all quantities are evaluated at $c=1/2$. 
It is convenient to express this in terms of the ratio $\alpha := \frac{\ado(1/2)}{a(1/2)-d(1/2)}$:
\be \label{adt-eq2}
\adt(1/2)=-2a(1/2) \alpha^2-4\alpha(3a(1/2)-d(1/2) )-8a(1/2).
\ee
We then obtain $\ddt(1/2)$ from equation (\ref{Edt-eq}):
\bea \label{ddt-eq}
\ddt(1/2)&=&-\adt(1/2)-8\ado(1/2) -8a(1/2)+8d(1/2)\cr
&=&2a(1/2) \alpha^2 +4\alpha (a(1/2)+d(1/2)) +8d(1/2).
\eea

Computing higher derivatives of $f$ involves organizing many terms. We use the notation $x_i$ to denote one of $\{a,b,d\}$,
and $f_i$ to denote a partial derivative with respect to $i$. Since $f$ is linear in $c$, we need only take one partial derivative in 
the $c$ direction, but arbitrarily many in the other directions. 

\be \label{fdo-eq}
\fdo=f_c +\sum_i f_i \xdo_i.
\ee

\be
\fdt=2\sum_i f_{ci}\xdo_i + \sum_i f_i\xdt_i + \sum_{ij}f_{ij} \xdo_i \xdo_j,
\ee
since $f_{cc}=0$. 

\bea
\fdth=&=&3\sum_i f_{ci} \xdt_i + 3\sum_{ij} f_{cij} \xdo_i \xdo_j
+\sum_i f_i\xdth_i
+3\sum_{i,j} f_{ij}\xdt_i \xdo_j +\sum_{ijk} f_{ijk}\xdo_i \xdo_j
\xdo_k\cr
&=&3f_{ca}\adt +3f_{cb}\bdt +3f_{cd} \ddt +3f_{caa}\ado^2+
+6f_{cab}\ado \bdo \cr
&&+6f_{cad}\ado \ddo +3f_{cbb}\bdo^2 +6f_{cbd}\bdo \ddo +3f_{cdd}\ddo^2
  \cr
&&+f_a\adth +f_b\bdth +f_d\ddth +3f_{aa}\ado \adt +3f_{ab}\ado \bdt
  +3f_{ad}\ado \ddt\cr
&&+3f_{ab}\bdo \adt +3f_{bb}\bdo \bdt +3f_{bd}\bdo \ddt +3f_{ad}\adt
  \ddo +3f_{bd}\bdt \ddo\cr
&&+3f_{dd}\ddt \ddo +f_{aaa}\ado^3 +3f_{aab}\ado^2 \bdo
  +3f_{aad}\ado^2 \ddo\cr
&&+3f_{abb}\ado \bdo^2 + 6f_{abd}\ado \bdo \ddo +3f_{add}\ado \ddo^2
  +3f_{bbd}\bdo^2 \ddo\cr
&&+3f_{bdd}\bdo \ddo^2 +f_{ddd}\ddo^3 +f_{bbb}\bdo^3
\eea

When $c=1/2$ we have $\ddo=\ddth=0$, so many of the terms vanish. We then have 
\bea
\fdth(1/2)&=&3f_{ca}\adt +3f_{cb}\bdt+3f_{cd}\ddt +3f_{caa}\ado^2
+6f_{cab}\ado \bdo\cr
&&+3f_{cbb}\bdo^2 +f_a \adth +f_b \bdth +3f_{aa}\ado \adt +3f_{ab}\ado
\bdt\cr
&&+3f_{ad}\ado \ddt+3f_{ab}\bdo \adt +3f_{bb}\bdo \bdt +3f_{bd}\bdo
\ddt\cr
&&+f_{aaa}\ado^3 +3f_{aab}\ado^2 \bdo +3f_{abb}\ado \bdo^2
+f_{bbb}\bdo^3 \cr
&=&(3f_{ca}+3f_{cb})\adt+3f_{cd}\ddt+(3f_{caa}-6f_{cab}+3f_{cbb})\ado^2\cr
&&+(f_a -f_b )\adth + (3f_{aa}+3f_{ab}-3f_{ab}-3f_{bb})\ado \adt\cr
&&+(3f_{ad}-3f_{bd})\ado
  \ddt+(f_{aaa}-3f_{aab}+3f_{abb}-f_{bbb})\ado^3,
\eea
evaluated at $c=1/2$, where in the last step we have also used $\bdo(1/2)=-\ado(1/2)$, 
$\bdt(1/2)=\adt(1/2)$ and  $\bdth(1/2)=-\adth(1/2)$. Solving for $\adth(1/2)$ then gives
\bea \label{adth-eq1}
\adth(1/2)&= \Big[ & 3(f_{ca}+f_{cb})\adt+3f_{cd}\ddt+3(f_{caa}-2f_{cab}+f_{cbb})\ado^2\cr
&&+3(f_{aa}-f_{bb})\ado \adt +3(f_{ad}-f_{bd})\ado \ddt 
\cr && +(f_{aaa}-3f_{aab}+3f_{abb}-f_{bbb})\ado^3 \Big ]/(f_b -f_a),
\eea
evaluated at $c=1/2$. What remains is do compute the partial derivatives of $f$ that appear in 
equation (\ref{adth-eq1}).

This is a long but straightforward exercise in calculus:
\bea
f&=&\phantom{+}\Spo(a)[cd(a-d)-(1-c)b(b-d)]\cr
&&+\Spo(b)[ca(a-d)-(1-c)d(b-d)]\cr
&&+\Spo(d)[c(d^2-a^2)+(1-c)(b^2-d^2)]
\eea

\bea
f_a&=&\phantom{+}\Spt(a)[cd(a-d)-(1-c)b(b-d)]\cr
&&+\Spo(a)(cd)+\Spo(b)(2ac-cd)+\Spo(d)(-2ac) \cr \cr 
f_b&=&\phantom{+}\Spt(b)[ca(a-d)-(1-c)d(b-d)]\cr
&&+\Spo(a)[(1-c)d-2(1-c)b]+\Spo(b)[-(1-c)d]+\Spo(d)[2(1-c)b]
\cr \cr 
f_c&=&\phantom{+}\Spo(a)[ad-d^2+b^2-bd]+\Spo(b)[a^2-ad+bd-d^2]\cr
&&+\Spo(d)[2d^2-a^2-b^2]
\cr\cr
f_d&=&\phantom{+}\Spt(d)[c(d^2-a^2)+(1-c)(b^2-d^2)]\cr
&&+\Spo(a)[-2cd+ac+(1-c)b]+\Spo(b)[-ac+2(1-c)d-(1-c)b]\cr
&&+\Spo(d)[2cd-2(1-c)d]
\eea

\bea
f_{ac}&=&\Spt(a)[ad-d^2+b^2-bd]+d\Spo(a)+(2a-d)\Spo(b)-2a\Spo(d)
\cr\cr
f_{bc}&=&\Spt(b)[a(a-d)+d(b-d)]+\Spo(a)(-d+2b)+\Spo(b)d-2b\Spo(d)
\cr\cr
f_{dc}&=&\Spt(d)[2d^2-a^2-b^2]+\Spo(a)[a-b-2d]+\Spo(b)[-a-2d+b]+\Spo(d)4d
\cr\cr
f_{aa}&=&\Spth(a)[cd(a-d)-(1-c)b(b-d)]+\Spt(a)(2cd)+2c\Spo(b)-2c\Spo(d)
\cr\cr
f_{bb}&=&\phantom{+}\Spth(b)[ca(a-d)-(1-c)d(b-d)]\cr
&&+\Spt(b)(-2(1-c)d)-2(1-c)\Spo(a)+2(1-c)\Spo(d)
\cr\cr
f_{ad}&=&\Spt(a)[-2cd+ca+(1-c)b]+c\Spo(a)-c\Spo(b)-2ac\Spt(d)
\cr\cr
f_{bd}&=&\phantom{+}\Spt(b)[-ac+2(1-c)d-(1-c)b]+(1-c)\Spo(a)-(1-c)\Spo(b)\cr
&&+2(1-c)b\Spt(d)
\eea

\bea
f_{caa}&=&\Spth(a)[d(a-d)+b(b-d)]+2d\Spt(a)+2\Spo(b)-2\Spo(d)
\cr\cr
f_{cab}&=&\Spt(a)(2b-d)+(2a-d)\Spt(b)
\cr\cr
f_{cbb}&=&\Spth(b)[a(a-d)+d(b-d)]+2d\Spt(b)+2\Spo(a)-2\Spo(d)
\cr\cr
f_{aaa}&=&\Spf(a)[cd(a-d)-(1-c)b(b-d)]+3cd\Spth(a)
\cr\cr
f_{aab}&=&\Spth(a)[-2(1-c)b+(1-c)d]+2c\Spt(b)
\cr\cr
f_{abb}&=&\Spth(b)(2ac-dc)-2(1-c)\Spt(a)
\cr\cr
f_{bbb}&=&\Spf(b)[ca(a-d)-(1-c)d(b-d)]-3(1-c)d\Spth(b)
\eea

We now compute the relevant terms at $c=1/2$. (The right hand side of equations (\ref{fdots-1}-\ref{ddf-eq}) are all 
intended to be evaluated at $c=1/2$.)
\bea \label{fdots-1}
f_a(1/2)&=&\Spt(a)[(1/2)d(a-d)-(1/2)a(a-d)]+(1/2)d\Spo(a)\cr
&+&(a-d/2)\Spo(a)-a\Spo(d)\cr
&=&\Spt(a)[-(1/2)(a-d)^2]+a[\Spo(a)-\Spo(d)]
\cr\cr
f_b(1/2)&=&\Spt(a)[(1/2)a(a-d)-(1/2)d(a-d)]+\Spo(a)[(d/2)-b]\cr
&+&\Spo(a)[-(d/2)]+b\Spo(d)\cr
&=&(1/2)\Spt(a)(a-d)^2-a[\Spo(a)-\Spo(d)]\cr
&=&-f_a(1/2),
\eea
so
\be
f_b(1/2)-f_a(1/2)=\Spt(a)(a-d)^2-2a[\Spo(a)-\Spo(d)].
\ee
Note that this is the same as the denominator in the formula (\ref{ado-eq}) for $\ado(1/2)$. 

\bea
f_{ca}(1/2)= f_{cb}(1/2) &=& \Spt(a)(a^2-d^2)+2a[\Spo(a)-\Spo(d)],
\cr
f_{cd}(1/2)&=&2\Spt(d)(d^2-a^2)+4d[\Spo(d)-\Spo(a)]
\cr
f_{caa}(1/2)=f_{cbb}(1/2)&=&\Spth(a)[a^2-d^2]+2d\Spt(a)\cr
&+&2[\Spo(a)-\Spo(d)]
\cr
f_{cab}(1/2)&=&2\Spt(a)(2a-d)
\cr
f_{caa}(1/2)-2f_{cab}(1/2)+f_{cbb}(1/2)&=&2\Spth(a)(a^2-d^2)+8(d-a)\Spt(a) \cr && +4[\Spo(a)-\Spo(d)]
\eea

\bea
f_{aa}(1/2)=-f_{bb}(1/2)&=&\Spth(a)[-(1/2)(a-d)^2]+d\Spt(a)+\Spo(a)-\Spo(d)
\cr
f_{aa}(1/2)-f_{bb}(1/2)&=&-(a-d)^2\Spth(a)+2d\Spt(a)+2[\Spo(a)-\Spo(d)]
\cr
f_{ad}(1/2)=-f_{bd}(1/2)&=&(a-d)\Spt(a)-a\Spt(d)
\cr
f_{ad}(1/2)-f_{bd}(1/2)&=&2(a-d)\Spt(a)-2a\Spt(d)
\eea

\bea 
f_{aaa}(1/2)=-f_{bbb}(1/2)&=&\Spf(a)[-(1/2)(a-d)^2]+(3/2)d\Spth(a)
\cr
f_{aab}(1/2)=-f_{abb}(1/2)&=&(1/2)\Spth(a)(d-2a)+\Spt(a)
\cr
(f_{aaa}-3f_{aab}+3f_{abb}-f_{bbb})(1/2) &=&
-\Spf(a)(a-d)^2+6a\Spth(a)\cr
&-&6\Spt(a)
\eea

Plugging all of these expressions back into equation (\ref{adth-eq1}) yields
\bea \label{adth-eq2}
\adth(1/2) &=&\Big \{[6\Spt(a)(a^2-d^2)+12a(\Spo(a)-\Spo(d))]\adt\cr
&&+[6\Spt(d)(d^2-a^2)+12d(\Spo(d)-\Spo(a))]\ddt\cr
&&+[6\Spth(a)(a^2-d^2)+24(d-a)\Spt(a)+12(\Spo(a)-\Spo(d))]\ado^2\cr
&&+[-3(a-d)^2\Spth(a)+6d\Spt(a) +6(\Spo(a)-\Spo(d))]\ado \adt\cr
&&+6[(a-d)\Spt(a)-a\Spt(d)]\ado \ddt\cr
&&+[-\Spf(a)(a-d)^2+6a\Spth(a)-6\Spt(a)]\ado^3 \Big \}\cr
&& * \left [\Spt(a)(a-d)^2-2a[\Spo(a)-\Spo(d)\right ]^{-1}
\eea

Finally, we need $\adf(1/2)$ and $\ddf(1/2)$. From equations (\ref{Edots}) and (\ref{Tdf-eq}), 
\bea
0=2\Edf(1/2)&=& \adf+\ddf+16\adth+48(\adt-\ddt)
\cr\cr
0=\frac43\Tdf(1/2)&=&(a^2+d^2)\adf+2ad\ddf+8a\ado \adth +6a\adt^2+12\ado^2
\adt+12d\ddt \adt +6a\ddt^2\cr
&&+8[3a^2 \adth +18a\ado \adt +6\ado^3 +\adth d^2 +6\ado d\ddt]\cr
&&+48[3a^2\adt +6a\ado^2 -d^2\adt -2ad\ddt]+64[3a^2\ado -3d^2\ado]
\eea

\bea
0&=&(4/3)\Tdf(1/2) -4ad\Edf(1/2) 
\cr
&=&(a-d)^2\adf +8a\ado \adth +6a\adt^2 +12\ado^2 \adt \cr
&&+12d\ddt \adt +6a\ddt^2\cr
&&+\adth(24a^2+8d^2-32ad)+144a\ado \adt +48\ado^3 +48\ado d\ddt\cr
&&+48[\adt(3a^2-2ad-d^2)+6a\ado^2]+192(a^2-d^2)\ado
\eea
Solving for $\adf$ gives
\bea \label{adf-eq} \adf(1/2) & = & - \Big [8a\ado \adth +6a\adt^2 +12\ado^2 \adt 
+12d\ddt \adt +6a\ddt^2\cr
&&+\adth(24a^2+8d^2-32ad)+144a\ado \adt +48\ado^3 +48\ado d\ddt\cr
&&+48[\adt(3a^2-2ad-d^2)+6a\ado^2]\cr
&&+192(a^2-d^2)\ado \Big ]/(a-d)^2.
\eea
We then obtain
\be \label{ddf-eq}
\ddf(1/2)=-\adf -16\adth -48(\adt-\ddt)
\ee

\subsection{Derivatives of $S$}

Finally, we compute derivatives of the functional $S$.

\be
S=c^2 S_o(a)+(1-c)^2S_o(b)+2c(1-c)S_o(d)
\ee

\bea
\Sdo&=&2cS_o(a)-2(1-c)S_o(b)+2(1-2c)S_o(d)\cr
&&+c^2\Spo(a)\ado +(1-c)^2\Spo(b)\bdo +2c(1-c)\Spo(d)\ddo
\eea

\bea
\Sdt&=&2[S_o(a)+S_o(b)-2S_o(d)]+4c\Spo(a)\ado -4(1-c)\Spo(b)\bdo \cr
&+&4(1-2c)\Spo(d)\ddo+c^2[\Spo(a)\adt +\Spt(a)\ado^2]+(1-c)^2[\Spo(b)\bdt+\Spt(b)\bdo^2]\cr
&+&2c(1-c)[\Spo(d)\ddt +\Spt(d) \ddo^2]
\eea

\bea
\Sdth&=&6[\Spo(a)\ado +\Spo(b)\bdo -2\Spo(d)\ddo]\cr
&&+6c[\Spo(a)\adt +\Spt(a)\ado^2]-6(1-c)[\Spo(b)\bdt+\Spt(b)\bdo^2]\cr
&&+6(1-2c)[\Spo(d)\ddt +\Spt(d)\ddo^2]\cr
&&+c^2[\Spo(a)\adth +3\Spt(a)\ado \adt +\Spth(a)\ado^3]\cr
&&+(1-c)^2[\Spo(b)\bdth +3\Spt(b)\bdo \bdt +\Spth(b)\bdo^3]\cr
&&+2c(1-c)[\Spo(d)\ddth +3\Spt(d)\ddo \ddt +\Spth(d)\ddo^3]
\eea

\bea
\Sdf&=&12[\Spo(a)\adt +\Spt(a)\ado^2 +\Spo(b)\bdt +\Spt(b)\bdo^2
-2\Spo(d)\ddt -2\Spt(d)\ddo^2]\cr
&&+8c[\Spo(a)\adth +3\Spt(a)\ado \adt +\Spth(a)\ado^3]\cr
&&-8(1-c)[\Spo(b)\bdth +3\Spt(b)\bdo \bdt +\Spth(b)\bdo^3]\cr
&&+8(1-2c)[\Spo(d)\ddth+3\Spt(d)\ddo \ddt+\Spth(d)\ddo^3]\cr
&&+ c^2[\Spo(a)\adf+4\Spt(a)\ado \adth+3\Spt(a)\adt^2 +6\Spth(a)\ado^2
\adt+\Spf(a)\ado^4]\cr
&&+(1-c)^2[\Spo(b)\bdf+4\Spt(b)\bdo \bdth+3\Spt(b)\bdt^2
+6\Spth(b)\bdo^2 \bdt+\Spf(b)\bdf]\cr
&&+2c(1-c)[\Spo(d)\ddf+4\Spt(d)\ddo \ddth+3\Spt(d)\ddt^2\cr
&&+6\Spth(d)\ddo^2 \ddt+\Spf(d)\ddo^4]
\eea

Evaluating at $c=1/2$ and using symmetry:
\bea\label{EQ:S2 S4}
\Sdt(1/2)&=&4[S_o(a)-S_o(d)]+4\Spo(a)\ado\cr
&&+(1/2)[\Spo(a)\adt+\Spt(a)\ado^2 +\Spo(d)\ddt]
\cr\cr
\Sdf(1/2) &=&24[\Spo(a)\adt+\Spt(a)\ado^2 -\Spo(d)\ddt]\cr
&&+8[\Spo(a)\adth+3\Spt(a)\ado \adt+\Spth(a)\ado^3])\cr
&&+(1/2)[\Spo(a)\adf+4\Spt(a)\ado \adth+3\Spt(a)\adt^2
+6\Spth(a)\ado^2 \adt +\Spf(a)\ado^4]\cr
&&+(1/2)[\Spo(d)\ddf+3\Spt(d)\ddt^2],
\eea
where the right hand sides of both equations are evaluated at $c=1/2$, using equations~\eqref{ado-eq},~\eqref{adt-eq2},~\eqref{ddt-eq},~\eqref{adth-eq2},~\eqref{adf-eq}, and~\eqref{ddf-eq}.

\section{Expansion near the triple 
point $(\E,\T)=\left (\frac12,\frac18
\right )$}
\label{SEC:triple}

The phase transition between phases I and II occurs where $\ddot S=0$. This curve intersects the ER curve $\T=\E^3$ at $(\E,\T)=\left ( \frac12, \frac18 \right)$. This is a triple point, where
phases I, II and III meet. In this section we prove

\begin{theorem} \label{Thm:Cubic} 
Near $(\E,\T)= \left (\frac12, \frac18 \right )$, the 
boundary between the symmetric and asymmetric bipodal phases
takes the form $\T = \E^3 - 8 \left (\E - \frac12 \right )^3
+ O\left ( \E - \frac12 \right )^4.$  
\end{theorem}

\begin{proof} Let 
\begin{eqnarray}
A(a,d) & = & \frac12 \Spt(a) (a-d)^2 - a\left ( \Spo(a) - \Spo(d) \right ).\cr
B(a,d) & = & - 2 (a+d) \left ( \Spo(a) - \Spo(d) \right ).\cr
C(a,d) & = & 4 \left ( S_o(a) - a S_o'(a) - S_o(d) + d S_o'(d) \right ).
\end{eqnarray}
Then 
\be
\alpha = \frac{-B}{2A},
\ee
and our formula for $\ddot S$ works out to 
\be \ddot S = A \alpha^2 + B \alpha + C = - \frac{B^2-4AC}{4A}. \ee
As long as $A \ne 0 $ (i.e. $d \ne a$), 
the phase transition occurs precisely where the discriminant
$\Delta = B^2-4AC$ vanishes. 

We now do series expansions for $A$, $B$ and $C$ in powers of 
$\delta a := a - \frac12$ and $\delta d := d - \frac12$. We begin with 
the Taylor Series for $S_o(z)$ and its derivatives.
\begin{eqnarray}
S_o(z) & = & \frac{\ln(2)}{2} -  \sum_{n=0}^\infty 
\frac{4^n \left (z -\frac12 \right )^{2n+2}}
{(n+1)(2n+1)} \cr 
S_o'(z) & = & -2 \sum_{n=0}^\infty \frac{4^n \left (z-\frac12 \right )^{2n+1}}
{2n+1} \cr 
S_o''(z) & = & -2 \sum_{n=0}^\infty 4^n \left (z - \frac12 \right )^{2n}.
\end{eqnarray}
(To derive these expressions, expand 
\be S_o''(z) = \frac{-1}{2} \left ( \frac{1}{z} + \frac {1}{1-z}
\right ) = \frac{-2}{1-4\left ( z-\frac12 \right)^2}
\ee
as a geometric series, and then integrate term by term to obtain the series
for $S_o'(z)$ and $S_o(z)$.)  
Plugging these expansions into the formula for $A$ yields
\be A = -(\delta a - \delta d)^2 \sum_{n=0}^\infty 4^n (\delta a)^{2n}
+ (1+2\delta a) \sum_{n=0}^\infty \frac{4^n}{2n+1} \left ( 
\delta a^{2n+1} - \delta d^{2n+1} \right ) := \sum_{m=1}^\infty A_m, \ee
where $A_m$ is a homogeneous $m$-th order polynomial in $\delta a$ and 
$\delta d$. The first few terms are:
\begin{eqnarray}
A_1 & = & \delta a - \delta d \cr 
A_2 & = & \delta a^2 - \delta d^2 \cr 
A_3 & = & \frac{4}{3} \left ( \delta a^3 - \delta d^3 \right ) \cr
A_4 & = & \frac43 \left ( - \delta a^4 + 6 \delta a^3\delta d -3 \delta a^2
\delta d^2 -2 \delta a \delta d^3 \right ).
\end{eqnarray}

We do similar expansions of $B$ and $C$:
\begin{eqnarray}
B & = & 4(a+\delta a + \delta d)\sum_{n=1}^\infty \frac{4^n \left (
\delta a^{2n+1} - \delta d^{2n+1} \right )}{2n+1} = \sum_{m=1}^\infty B_m \cr 
B_1 & = & 4(\delta a - \delta d) \cr 
B_2 & = & 4(\delta a^2 - \delta d^2) \cr 
B_3 & = & \frac{16}{3} \left ( \delta a^3 - \delta d^3 \right ) \cr
B_4 & = & \frac{16}3 \left ( \delta a^4 + \delta a^3 \delta d - 
\delta a \delta d^3 - \delta d^4 \right )
\end{eqnarray}

\begin{eqnarray}
C & = & 4 \sum_{n=0}^\infty \frac{4^n \left ( \delta a^{2n+1}-\delta d^{2n+1}\right)}
{2n+1} + 4 \sum_{n=0}^\infty \frac{4^n \left ( \delta a^{2n+2}-\delta d^{2n+2}\right)}
{n+1}  = \sum_{m=1}^\infty C_m \cr 
C_1 & = & 4(\delta a - \delta d) \cr 
C_2 & = & 4(\delta a^2 - \delta d^2) \cr 
C_3 & = & \frac{16}{3} \left ( \delta a^3 - \delta d^3 \right )
\cr
C_4 & = & 8 \left ( \delta a^4 - \delta d^4 \right ).
\end{eqnarray}
Note that $4A_m=B_m=C_m$ when $m=1$, 2, or 3. This implies that
the discriminant $\Delta$ vanishes through 4th order, and the leading
nonzero term is 
\begin{eqnarray}
\Delta_5 & = & 2B_1B_4 + 2B_2B_3 - 4A_1C_4-4A_2C_3 -4A_3C_2 - 4A_4C_1 \cr
&=& 2B_1 B_4 + 2 B_2B_3 -B_1C_4 -B_2B_3 -B_3B_2-4A_4C_1 \cr 
&=& B_1(2B_4-C_4-4A_4) \cr 
&=& \frac{8B_1}{3} \left (3 \delta a^4 - 8\delta a^3 \delta d 
+ 6 \delta a^2 \delta d^2 - \delta d^4 \right ) \cr 
&=& \frac{32}{3}(\delta a - \delta d)^4 (3 \delta a + \delta d). 
\end{eqnarray}
In fact, all terms in the expansion of $\Delta$ are divisible by 
$(\delta a -\delta d)^4$, as can be seen by evaluating $\Delta$ and its 
first three derivatives with respect to $\delta d$ at $\delta d=\delta a$.
We can view $3 \delta a + \delta d$ as the 
{\em leading} term in the expansion of 
$\frac{3\Delta}{32(\delta a - \delta d)^4}$. Setting $\Delta=0$ then gives
\be\delta d = -3 \delta a + O(\delta a^2).\ee 
Since 
\be \E = \frac{a+d}{2}; \qquad \T-\E^3 = \frac{(a-d)^3}{8}, \ee
we have 
\be \T = \E^3 -8 \left (\E - \frac12 \right )^3 + 
O ((\E-\frac12 )^4 ), \ee
as required. Note that these expressions only apply when $\delta a<0$, i.e.,
when $d>a$, i.e., for $\E > \frac12$. When $\E<\frac12$, bipodal graphons with
$\delta d = -3\delta a$ would lie {\em above} the ER curve. 

\end{proof}  
 
\section{Numerical values of $\ddot S$ and $\ddddot S$}
\label{SEC:Num Eval}

We now evaluate $\Sdt$ and $\Sdf$ numerically,
following the formulas given in~\eqref{EQ:S2 S4}, with help from formulas~\eqref{ado-eq},~\eqref{adt-eq2},~\eqref{ddt-eq},~\eqref{adth-eq2},~\eqref{adf-eq}, and ~\eqref{ddf-eq}.

\medskip

We show first in Figure~\ref{FIG:Boundary} the boundary of the symmetrical bipodal phase. For better visualization, we use the coordinate $(\E, \T-\E^3)$ instead of $(\E, \T)$ (in which the phase boundary curve bends too much to see the details). In this new coordinate, the ER curve becomes the curve $(\E, 0)$ which is the upper boundary in the plot. The part of the phase boundary on which $\ddot S=0$, denoted by $\Sigma$, is illustrated with thick red line.
\begin{figure}[!ht]
\centering
\includegraphics[angle=0,width=0.45\textwidth]{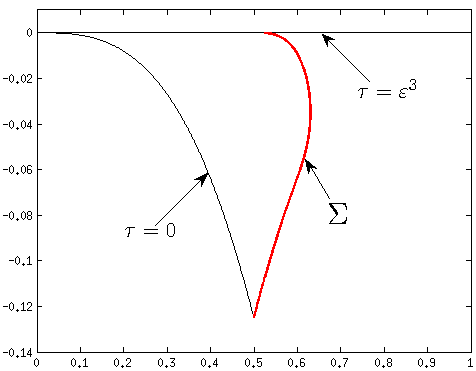}
\caption{The boundary of the bipodal phase in the $(\E, \T-\E^3)$ coordinate. The part of phase boundary on which $\ddot S=0$ is shown in the thick red line.
Note how this line is tangent to the ER curve at the triple point, as required
by Theorem~\ref{Thm:Cubic}.}
\label{FIG:Boundary}
\end{figure}

\begin{figure}[!ht]
\centering
\includegraphics[angle=0,width=0.35\textwidth]{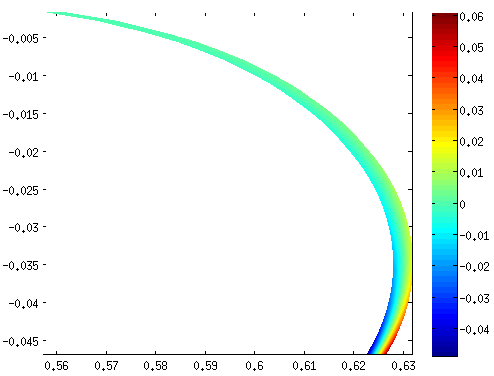} \hskip 1cm
\includegraphics[angle=0,width=0.35\textwidth]{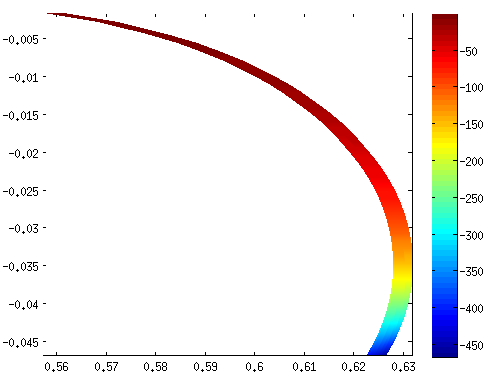}\\ 
\includegraphics[angle=0,width=0.35\textwidth]{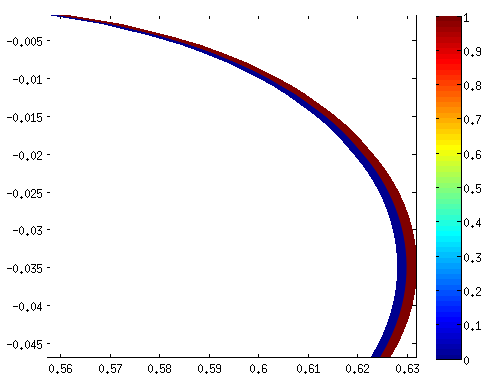}  \hskip 1cm
\includegraphics[angle=0,width=0.35\textwidth]{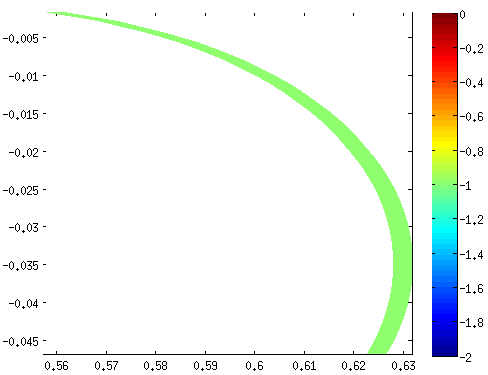} 
\caption{The function $\ddot S$ (top left), $\sgn(S)$ (bottom left), $\ddddot S$ (top right) and $\sgn(\ddddot S)$ (bottom right) in the neighborhood of the phase transition curve $\Sigma$.}
\label{FIG:S2 S4}
\end{figure}
In Figure~\ref{FIG:S2 S4}, we show the values of $\ddot S$ and $\ddddot S$ along a tube along the phase transition curve $\Sigma$. The plot is again in the $(\E, \T-\E^3)$ coordinate. To better visualize the transition, we visualize the function $\sgn(\ddot S)$ and $\sgn(\ddddot S)$ where $\sgn$ is the sign function: $\sgn(x)=1$ when $x\ge 0$ and $\sgn(x)=-1$ when $x<0$. The function $\dfrac{\partial\ddot S}{\partial \E}$ is show in Figure~\ref{FIG:pSpe} along the phase transition curve $\Sigma$.
\begin{figure}[ht]
\centering
\includegraphics[angle=0,width=0.45\textwidth]{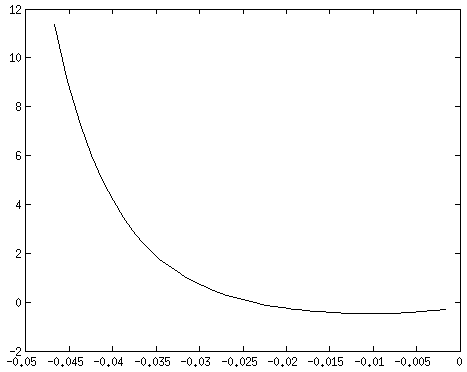} 
\caption{The plot of $\dfrac{\partial\ddot S}{\partial \E}$ on the curve $\Sigma$, i.e. as a function of $\T-\E^3$.}
\label{FIG:pSpe}
\end{figure}

The point is that $\ddddot S$ is always negative when $\ddot S=0$. Near
the transition, we should thus expect optimizing graphons to have 
$|c - \frac12| \approx \sqrt{-6 \ddot S(\frac12)/\ddddot S(\frac12)}$
when $\ddot S(\frac12)>0$, and to have $c = \frac12$ when $\ddot S(\frac12)<0$. 
In the next section we confirm this prediction with direct sampling of 
graphons. 
 
\section{Comparison to numerical sampling}
\label{SEC:Simulation}

In this section, we perform some numerical simulations using the sampling 
algorithm we developed in~\cite{RRS}, and compare them to the results of the 
perturbation analysis. The sampling algorithm construct random samples of 
values of $\mathcal S$ in the parameter space graphons, and then take the maximum of 
the sampled values. This sampling algorithm is extremely expensive 
computationally, but when sufficiently large sample size are reached, we can 
achieve desired accuracy; see the discussions in~\cite{RRS}. We emphasize that 
our sampling algorithm does not assume bipodality of the maximizing graphons. 
In fact, we always start with the assumption that the graphons are $16$-podal. 
Bipodal structures are found when all the other clusters have size zero.

\medskip

\begin{figure}[ht]
\centering
\includegraphics[angle=0,width=0.45\textwidth]{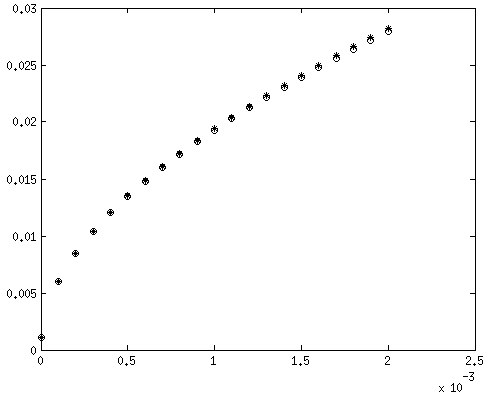} \hskip 1cm
\includegraphics[angle=0,width=0.45\textwidth]{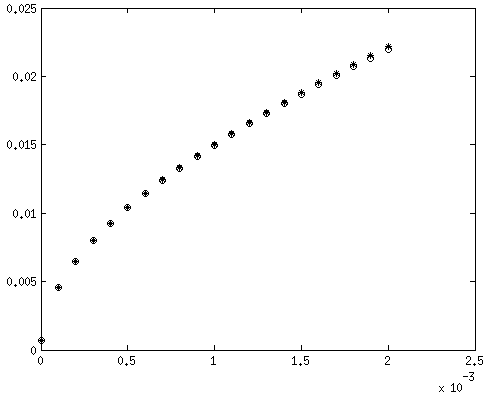}\\ 
\caption{Comparison between the curve $(\E, |c-\frac{1}{2}|)$ (in $\circ$, given by the sampling algorithm) and $(\E, \sqrt{-6\ddot S(\frac12)/\ddddot S(\frac12)})$ (in $*$, given by the perturbation calculation) at two different $\T-\E^3$ values: $-0.0348$ (left), $-0.0300$ (right).}
\label{FIG:Cross}
\end{figure}
In Figure~\ref{FIG:Cross}, we plot the values of $|c - \frac12|$ for optimizing graphons (found with the numerical sampling algorithm) and the real part of $\sqrt{-6\ddot S(\frac12)/\ddddot S(\frac12)}$ (given by the perturbation analysis) as functions of $\E$ along at two different $\T-\E^3$ values. It is easily seen that the perturbation calculation gives very good fit when we are reasonably close to the phase transition curve $\Sigma$, but starts to be less accurate when we go farther out. A similar plot for the corresponding values of $a$, $b$, and $d$ are shown in Figure~\ref{FIG:a b d versus E}. As expected, $a$ and $d$ show square root singularities, just like $c$, but $b$ does not.
\begin{figure}[!ht]
\centering
\includegraphics[angle=0,width=0.3\textwidth]{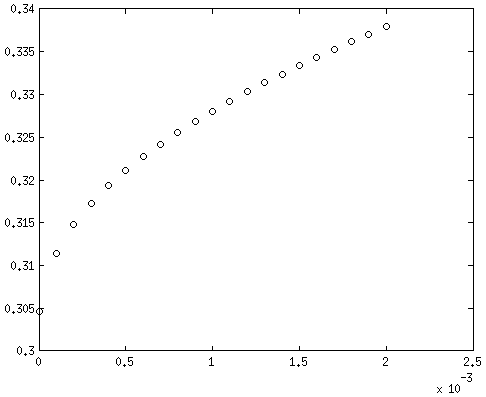} 
\includegraphics[angle=0,width=0.3\textwidth]{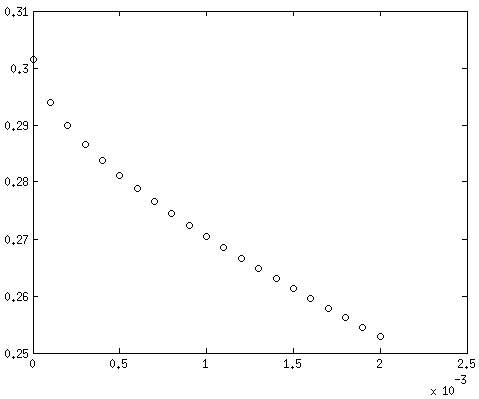} 
\includegraphics[angle=0,width=0.3\textwidth]{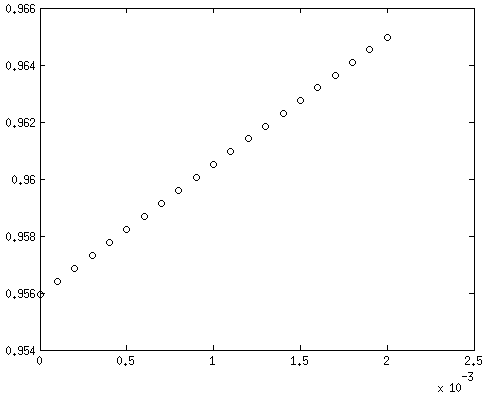}\\ 
\includegraphics[angle=0,width=0.3\textwidth]{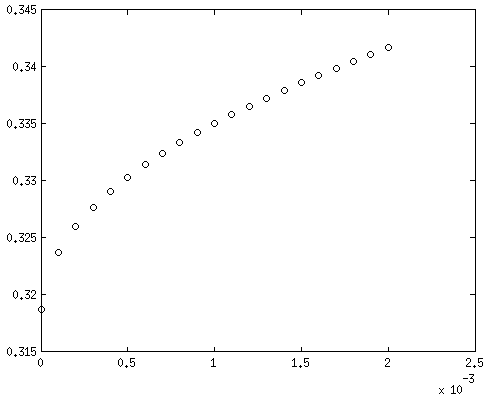} 
\includegraphics[angle=0,width=0.3\textwidth]{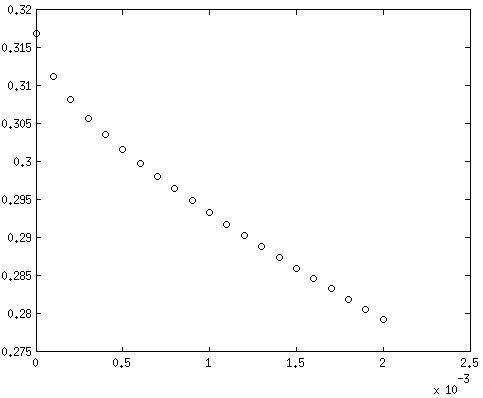} 
\includegraphics[angle=0,width=0.3\textwidth]{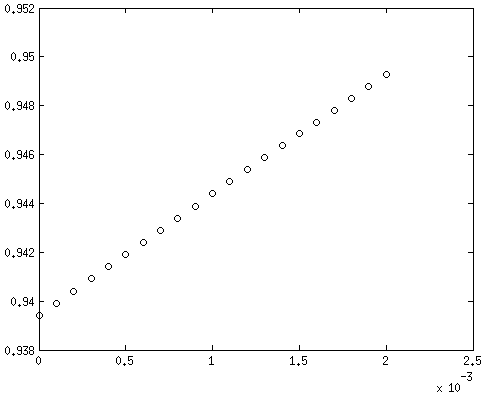}\\ 
\caption{The plots of $a$ (left), $b$ (middle) and $d$ (right) versus $\E-\E_0$ at two different values of $\T$: $0.2147$ (top row) and $0.2184$ (bottom row).}
\label{FIG:a b d versus E}
\end{figure}

\medskip

Finally, we show in Figure~\ref{FIG:Graphons} some typical graphons that we obtained using the sampling algorithm. We emphasize again that in the sampling algorithm, we did not assume bipodality of the optimizing graphons. The numerics indicate that the optimizing graphons are really bipodal close to (on both sides) the phase transition curve $\Sigma$. These serve as numerical evidence to justify the perturbation calculations in this work.  
\begin{figure}[ht]
\centering
\includegraphics[angle=0,width=0.4\textwidth]{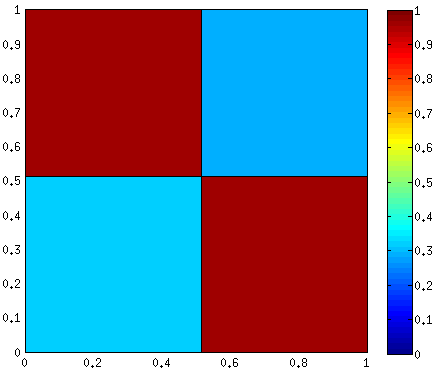} 
\includegraphics[angle=0,width=0.4\textwidth,height=0.336\textwidth]{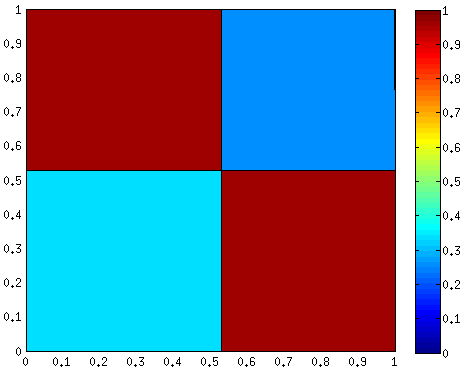}\\ 
\includegraphics[angle=0,width=0.4\textwidth]{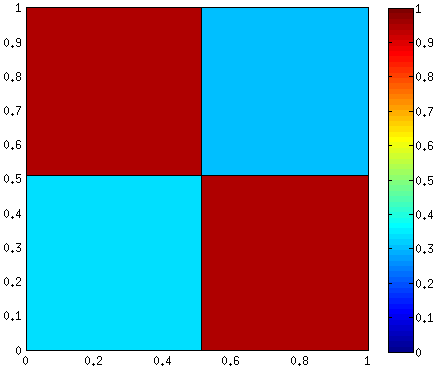} 
\includegraphics[angle=0,width=0.4\textwidth]{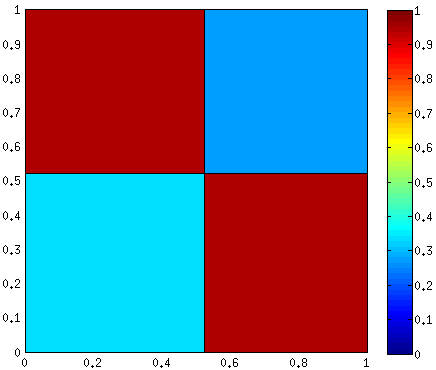} 
\caption{Typical graphons at points close to the phase boundary. The values of $(\E, \T)$ are (from top left to bottom right): $(0.6299,0.2147)$, $(0.6315, 0.2147)$, $(0.6290, 0.2184)$ and $(0.6306,0.2184)$.}
\label{FIG:Graphons}
\end{figure}

\section{Conclusion}
\label{SEC:Concl}

Our analysis began with the assumption that all entropy optimizers for
relevant constraint parameters are multipodal; in fact bipodal, but
we emphasize the more general feature. Multipodality is a
fundamental aspect of the asymptotics of constrained graphs. It is the
embodiment of phase emergence: the set of graphs with fixed constraints develops
(``emerges into'') a well-defined global state as the number of nodes
diverges, \emph{by partitioning the set of all nodes into a finite
(usually small) number of equivalent parts,} $\{P_1,P_2,\ldots\}$, of
relative sizes $c_j$, and uniform (i.e.\ constant) probability $p_{ij}$
of an edge between nodes in $P_i$ and $P_j$. This unexpected fact has
been seen in all simulations, and was actually proven throughout the phase spaces
of edge/$k$-star models~\cite{KRRS1} and in particular regions of a wide class of
models~\cite{KRRS2}. In this paper we are taking this as given, for a certain
region of the edge/triangle phase space; we are assuming that
our large constrained graphs emerge into such global states, and not
just multipodal states but bipodal states, in accordance with
simulation of the relevant constraint values in this model.

We analyzed the specific constraints of edge density approximately $1/2$ and triangle density less than $1/8$, and drew a
variety of conclusions. First we used the facts, proven in~\cite{RRS}, that
there is a smooth curve in the phase space, indicated in Figure~\ref{scallop}, such
that: (i) to the left of the curve the (reduced) entropy optimizer is
unique and is given by~\eqref{EQ:Graphon2}; (ii) the entropy optimizer is also unique to
the right of the curve but is no longer symmetric. The curve thus
represents a boundary between distinct phases of different
symmetry. These facts were established in~\cite{RRS}. The analysis in this
paper concerns the details of this transition. Here we want to speak
to the significance of the results, in terms of symmetry breaking.

We take the viewpoint that the global states for these constraint
values are sitting in a 4 dimensional space given by $a, b,
d$ and $c$, quantities which describe the laws of interaction
of the elements of the parts $P_j$ into which the node set is
partitioned. (The notion of `vertex type' has been introduced and
analyzed in \cite{Ko} and is useful in more general contexts.)

The symmetry $c_1=c_2=1/2$ and $a=b$ is therefore a symmetry
of the \emph{rules} by which the global state is produced. Each part
$P_j$ into which the set of nodes is partitioned can be thought of as
embodying the symmetry between the nodes it contains, but the symmetry
of the global states is a higher-level symmetry, between the way these equivalence
classes $P_j$ are connected. In this way the transition studied in
this paper is an analogue of the symmetry-breaking fluid/crystal
transition studied in the statistical analysis of matter in thermal
equilibrium~\cite{Ru2}.

One consequence is a contribution to the old problem concerning the
fluid/crystal transition. It has been known experimentally since the
work of PW Bridgman in the early twentieth century that no matter how one varies
the thermodynamic parameters (say mass and energy density) between
fluid and crystal phases one \emph{must} go through a singularity or
phase transition. An influential theoretical analysis by L. Landau
attributed this basic fact to the difference in symmetry between the
crystal and fluid states, but this argument has never been completely
convincing~\cite{Pip}.  In our model this fact follows for phases $II$
and $III$ from the two-step process of first simplifying the space of
global states to a finite dimensional space (of interaction rules) and
then realizing the symmetry as acting in that space; it is then
immediate that phases with different symmetry cannot be linked by a
smooth (analytic) curve.  This is one straightforward consequence of
the powerful advantage of understanding our global states as
multipodal, when compared say with equilibrium statistical mechanics.

\section*{Acknowledgments}

This work was partially supported by NSF grants DMS-1208191,
DMS-1509088, DMS-1321018, DMS-1101326 and PHY-1125915.



\end{document}